\documentclass{amsart}

\makeatletter
\@namedef{subjclassname@2020}{%
  \textup{2020} Mathematics Subject Classification}
\makeatother
\newtheorem{theorem}{Theorem}[section]

\usepackage{color}

\theoremstyle{definition}

\newtheorem{corollary}[theorem]{Corollary}

\theoremstyle{remark}
\newtheorem{remark}[theorem]{Remark}

\numberwithin{equation}{section}

\usepackage[]{draftwatermark}
\SetWatermarkText{DRAFT}
\SetWatermarkScale{5}

\begin{document}

\title[Generalized Parabolic Frequency on compact  manifolds]{ Generalized Parabolic Frequency on compact  manifolds}

\author{Shahroud Azami}
\address{Department of  Pure Mathematics, Faculty of Science, Imam Khomeini International University,
Qazvin, Iran. }

\email{azami@sci.ikiu.ac.ir}

\author{Abimbola Abolarinwa}
\address{Department of Mathematics, University
of Lagos, Akoka, Lagos State, Nigeria}
\email{a.abolarinwa1@gmail.com}




\begin{abstract}
In this paper,  we first prove monotonicity of a generalized parabolic frequency on weighted closed Riemannian manifolds for some linear heat equation. Secondly, a certain generalized parabolic frequency functional is defined with respect to the solutions of a nonlinear weighted $p$-heat-type equation on manifolds, and its monotonicity is proved. Notably, the monotonicities are derived with no assumption on both the curvature and the potential function. Further consequences of these monotonicity formulas from which we can get backward uniqueness are discussed.
\end{abstract}

\maketitle

\tableofcontents

\section{Introduction}

The following frequency functional was  first introduced by Almgren in \cite{A} 
\begin{align}\label{01}
N(r)=\frac{r\int_{B(p,r)}|\nabla v|^{2}dx}{\int_{\partial B(p,r)}v^{2}dA}
\end{align}
for a harmonic function $v$ on Euclidean space $\mathbb{R}^{n}$. In this quantity $p$ is  a fixed  point in $\mathbb{R}^{n}$,   $B(p,r)$ is a ball  of radius $r>0$ centred at $p$,   while  $\partial B(p,r)$ and $dA$ are  the boundary of $B(p,r)$  and the induced $(n-1)$-dimensional Hausdorff measure on $\partial B(p,r)$, respectively.   This functional in (\ref{01}) has played a crucial role in the analysis of linear and nonlinear elliptic partial differential equations,  starting with its monotonicity properties proved in \cite{A}.  Properties such as local regularity of harmonic functions,  unique continuation of  elliptic operators, and estimates of nodal sets of solutions, to name but a few have been proved with the aid of (\ref{01}).  Many of these results have been extended to various settings most especially Riemannian manfold setting. For detail discussion on applications of this functional and its various extensions,  and the generalization of these results interested readers  can see the following literature \cite{G1, G2,L,TH1,XL} and the references cited therein.
The counterpart of $N(r)$ for  the solution to the heat equation  on $\mathbb{R}^{n}$  is called parabolic frequency  functional, and  was  first introduced by Poon \cite{P} who used  it to study  the unique  continuation   of solutions  to  parabolic equations on  $\mathbb{R}^{n}$.   Later  Ni \cite{N} proved monotonicity of such parabolic frequency functional on Riemannian (or K\"ahler) manifold with nonnegative sectional  (or bisectional) curvature and parallel Ricci curvature for holomorphic function using Hamilton's matrix Harnack estimate.  

For a given  Riemannain manifold $(M^{n},g)$ with a potential function $f: M\to \mathbb{R}$ there is an associated second order diffusion operator called drifting or Witten Laplacian defined as follows
\begin{align}\label{02}
\Delta_{f}u={\rm div}_{f}(\nabla u)= e^{f}{\rm div}(e^{-f}\nabla u) =\Delta u-\langle \nabla f, \nabla u\rangle.
\end{align}
If the Riemannian volume form on $M$ (associated with metric $g$) is denoted by $d\nu$, then equipping $(M,g)$ with the weighted volume form $d\mu=e^{-f}d\nu$ makes the triple $(M,g,d\mu)$ a weighted manifold which is usually called  a smooth metric measure space.  It is clear  from (\ref{02}) that drifting Laplacian, $\Delta_{f}$, is a self-adjoint operator with respect to the weighted measure, $d\mu=e^{-f}d\nu$, in the sense that 
\begin{align}
\int_M (\Delta_{f} u) w d\mu = - \int_M \langle \nabla u, \nabla w\rangle d\mu = \int_M u  (\Delta_{f} w)d\mu.
\end{align}
Assume that $v: M\times[a,b]\to \mathbb{R}^n$ is smooth and solves the weighted heat equation 
\begin{align}\label{h}
v_t- \Delta_{f} v =0
\end{align}
 on $M$ such that $u,u_t \in W^{1,2}(M,d\mu)$, Colding and Minicozzi showed in \cite{TH} that the frequency functional $H(t)$ defined by 
\begin{align}
H(t):= - \frac{\int_M |\nabla v|^2 d\mu}{\int_M v^2 d\mu}
\end{align}
is monotonically nonincreasing  under (\ref{h}) with no restriction on the curvature of $M$, See \cite[Theorem 0.6]{TH}.  For a solution of (\ref{h}) Baldauf and Kim \cite{JB} defined the  parabolic frequency  involving some correction term as follows
$$U(t)= \frac{\tau(t)\langle \Delta_f u, u\rangle_{L^{2}(d\mu)}^{2}}{|| u||_{L^{2}(d\mu)}^{2}}e^{-\int \frac{1-k(t)}{\tau(t)}dt}, $$
where  $\tau(t)$ is the backwards time, $k(t)$ be the time-dependent function, and $d\mu$ is the
weighted measure. They proved that parabolic frequency $U(t)$ for the solution of weighted heat equation is monotone increasing along the Ricci flow with bounded Bakry-\'{E}mery Ricci curvature.  See also \cite{Ab, Az,LLX, LZ} for similar results on time-dependent Riemannian metrics.

However, we will consider two generalizations of (\ref{02}) in this paper.  First, for a symmetric positive  definite  $(1,1)$-tensor $T$ on $M$ and $f\in C^{\infty}(M)$ or atleast $C^2$, we define a generalized diffusion operator
\begin{equation}\label{05}
\mathcal{L}u={\rm div}_{f}(T(\nabla u))={\rm div}(T(\nabla u))-\langle \nabla f, T(\nabla u)\rangle
\end{equation}
 for any $C^{2}$ function $u$ on $M$.  Note that if $T=id$ (\ref{05}) reduces to (\ref{02}). Throughout this paper, for any vector field $X,Y$ on $M$, we denote $\langle T(X),Y\rangle$ by  $T(X,Y)$. From the properties of ${\rm div}_{f}$ and the symmetry of $T$ we have
\begin{equation*}
\mathcal{L}(uv)=v\mathcal{L}(u)+u\mathcal{L}(v)+2T(\nabla u,\nabla v),
\end{equation*}
for any $u,v\in C^{2}(M)$.
Let $\Omega\subset M$ be  a compact bounded domain  with smooth boundary $\partial \Omega$ in $M$. Let $dm$ be the volume form on the boundary induced by the outward normal vector field  $\overrightarrow n $ on  $\partial \Omega$. The divergence theorem for operator $\mathcal{L}$ gives
\begin{equation*}
\int_{\Omega}\mathcal{L}u d\mu=\int_{\partial \Omega}T(\nabla u,\overrightarrow n)dm.
\end{equation*}
Therefore the integration by parts formula is
\begin{equation*}
\int_{\Omega}v\mathcal{L}u d\mu=-\int_{\Omega}T(\nabla u, \nabla v)d\mu+\int_{\partial \Omega}vT(\nabla u,\overrightarrow n)dm.
\end{equation*}
Thus in the space of functions in $L^{2}(\Omega, d\mu)$ which vanish on  $\partial \Omega$ the operator $\mathcal{L}$  is a self-adjoint operator.
The  eigenvalues of the operator $\mathcal{L}$ are studied in \cite{GO, MA}. 

The second generalization is the so called weighted $p$-diffusion operator $\mathcal{L}_p$, defined by
\begin{align}\label{p}
\mathcal{L}_p : & = {\rm div}_f (|\nabla u|^{p-2}T(\nabla u))  =e^{f} {\rm div} (e^{-f} |\nabla u|^{p-2}T(\nabla u)) \nonumber \\
& =  {\rm div} (|\nabla u|^{p-2}T(\nabla u)) - |\nabla u|^{p-2} \langle\nabla f, T(\nabla u)\rangle.
\end{align}
When $T$ is the identity,  $\mathcal{L}_p $ is the weighted $p$-Laplacian, if in addition $p=2$, it is then the drifting Laplacian (\ref{02}). If $T$ is the identity and $f$ is constant then we have the usual $p$-Laplacian,  
$\Delta_p u= {\rm div}(|\nabla u|^{p-2}\nabla u)$. Note that the usual $p$-Laplacian is nonlinear in general,

Now we denote $u_t(x,t)$ as partial derivative of $u(x,t)$, and consider the following linear heat equation
\begin{equation}\label{1}
u_{t}(x,t)-\mathcal{L}u(x,t)=\phi(t)u(x,t),
\end{equation}
 where $\phi$  is a smooth function with respect to time-variable $t$, $u$  is a  smooth function on $M\times [a,b]$. We define the parabolic frequency  $U(t)$ for the  solution of heat equation  (\ref{1})    as follows
 \begin{equation}\label{2}
U(t)=\frac{D(t)}{I(t)},
\end{equation}
where
 \begin{equation}\label{3}
D(t)=-\int_{M}T(\nabla u, \nabla u)d\mu=\int_{M}u\mathcal{L}ud\mu,
\end{equation}
and
 \begin{equation}\label{4}
I(t)=\int_{M}u^{2}d\mu.
\end{equation}

To conclude this introduction, we highlight the plan of the rest of the paper in order to put our results in proper perspective.  In Section \ref{S2}, we prove monotonicity of $U(t)$ defined in (\ref{2}) under the influence of the linear heat equation (\ref{1}). The result which is stated in Theorem \ref{t1} is a generalization of \cite[Theorem 0.6]{TH}.  Our method of proof follows a similar approach to \cite{TH} but it is more involved, whereas \cite[Theorem 0.6]{TH} becomes a special case of ours for $T=id$ on $M$ and $\phi(t)\equiv 0$. An immediate consequence of this result is stated in Corollary \ref{cor1}. Furthermore, Theorem \ref{t2} and Corollary \ref{cor2} are presented to demonstrate that our results hold also for more general parabolic inequalities, in which case we have assumed
\begin{align*}
|u_{t}-\mathcal{L} u|\leq \psi(t)\left(|u|+\sqrt{T(\nabla u, \nabla u)}\right)
\end{align*}
with  $\psi(t)$ being a time dependent constant function.

In the final part of this paper (Section \ref{S3}), a nonlinear $p$-heat-type equation involving weighted $p$-Laplacian,   $\Delta_{f,p}$, $p>1$ (where we have taken $T=id$ in (\ref{p})), is considered
\begin{align}\label{010}
|u|^{p-2}u_{t}-\Delta_{f,p} u  =\eta(t)|u|^{p-2}u
\end{align}
 on $M$, where  $\Delta_{f,p} : = {\rm div}_f (|\nabla u|^{p-2}\nabla u)$, $\eta$  is a smooth function with respect to time-variable $t$ and $u$  is a  smooth function on $M\times [a,b]$.   The $p$-Lpalacian heat-equation of the form 
 $$|u|^{p-2}u_{t}=\Delta_{p} u$$
is well studied in literature. For instance, in  \cite{BK, XZ} Harnack inequalities for its positive solutions were obtained, while Wang \cite{FW} studied gradient estimates on its positive smooth solution along the Ricci flow. Recently,  Xu, Shen and Wang in \cite{XX} discussed parabolic frequency monotonicity for $|u|^{p-2}u_{t}=\Delta_{p} u$ on a connected finite graph. For our case a corresponding parabolic frequency functional $U_p(t)$ is given on $M$ as
\begin{align}
U_p(t) := \frac{\int_M u \Delta_{f,p}u e^{-f} d\nu}{\int_M |u|^p e^{-f} d\nu}
\end{align} 
with potential function $f$ (see (\ref{sa2})).  Here we establish the monotonicity of $U_p(t)$ along the weighted $p$-heat-type equation (\ref{010})  (see Theorem \ref{t3}), and then discuss some of its consequences, most especially to deriving backward uniqueness for the $p$-heat equation on compact manifolds.  Our results in this section extends recent results obtained in \cite{XX}  for graph on one hand,  and that of \cite{TH} for $p=2$ on the other hand.

\section{Generalized parabolic frequency}\label{S2}
\subsection{Parabolic frequency on weighted heat-type equation}
Our first result in this section is the following.
\begin{theorem}\label{t1}
Let $u: M\times [a,b] \to \mathbb{R}^n$ solve the weighted heat-type equation (\ref{1}), then $U(t)$  (defined in (\ref{2}) ) is a nondecreasing function. Moreover, if $U$ is constant, then $u$ is an eigenfunction of $f$  and
\begin{equation}\label{azs}
u(x,t)=e^{tU(t)-aU(a)+\int_{a}^{t}\phi(\tau)d\tau}u(x,a).
\end{equation}
Also, if $\phi(t)$ is a nondecreasing function then $\log I$ is convex.
\end{theorem}
\begin{proof}
Direct computing  and  integrating by parts gives
\begin{eqnarray}\nonumber
I'(t)&=&2\int_{M}uu_{t}d\mu\\\nonumber&=&2\int_{M}u(\mathcal{L} u+\phi u)d\mu\\\nonumber&=&-2\int_{M}T(\nabla u,\nabla u)d\mu+2\phi(t)\int_{M}u^{2}d\mu\\\label{5}&=&2\phi(t)I(t)+2D(t),
\end{eqnarray}
and
\begin{eqnarray}\nonumber
D'(t)&=&-2\int_{M}T(\nabla u,\nabla u_{t})d\mu\\\nonumber&=&-2\int_{M}T(\nabla u,\nabla( \mathcal{L}u+\phi(t) u))d\mu\\\label{6}&=&2\int_{M}( \mathcal{L}u)^{2}d\mu+2\phi(t)D(t).
\end{eqnarray}
Using (\ref{5}) and the definition of $U$ we obtain
\begin{equation}
(\log I)'(t)=2\phi(t)+2\frac{D(t)}{I(t)}=2\phi(t)+2U(t).
\end{equation}
Thus, applying (\ref{5}), (\ref{6}), and definition of $D(t)$ we infer
\begin{eqnarray}\nonumber
D'(t)I(t)-D(t)I'(t)&=&\left(2\int_{M}( \mathcal{L}u)^{2}d\mu+2\phi(t)D(t)\right)\left(\int_{M}u^{2}d\mu\right)\\\nonumber&&-D(t)\left(2\phi(t)I(t)+2D(t)\right)\\\label{7}&=&
2\left(\int_{M}( \mathcal{L}u)^{2}d\mu\right)\left(\int_{M}u^{2}d\mu\right)-2\left( \int_{M}u\mathcal{L}u \,d\mu\right)^{2}\geq0.
\end{eqnarray}
The last inequality is due to the Cauchy-Schwarz inequality. From this we conclude that
\begin{equation}
U'(t)=\frac{D'(t)I(t)-D(t)I'(t)}{I^{2}(t)}\geq0.
\end{equation}
Therefore, $U(t)$ is a nondecreasing function. If  $U$ is constant then $U'(t)=0$ and the equality in the Cauchy-Schwarz inequality ( \ref{7}) implies that
\begin{equation}
\mathcal{L}u=c(t)u,
\end{equation}
for  some  time-dependent smooth function $c$. Hence, $u$ is an eigenfunction of $\mathcal{L}$.  By definition of $D(t)$ we deduce
\begin{equation}
D(t)=\int_{M}u\mathcal{L}ud\mu=c(t)\int_{M}u^{2}d\mu=c(t)I(t).
\end{equation}
It follows that $c(t)=U(t)$ and $\mathcal{L}u=U(t) u$. Let $v(x,t)=e^{-Ut}u(x,t)$. Hence,
\begin{equation}
v_{t}(x,t)=e^{-Ut}\left(-Uu+u_{t}\right)=e^{-Ut}\left(-Uu+\mathcal{L}u+\phi(t) u\right)=\phi(t)v(x,t).
\end{equation}
By solving the last equation we arrive at (\ref{azs}).

\end{proof}

An immediate consequence of the above theorem is given in the next corollary.
\begin{corollary}\label{cor1}
Let $u: M\times [a,b] \to \mathbb{R}^n$ solve the heat equation (\ref{1}), then
\begin{equation}
I(b)\geq I(a)e^{2\int_{a}^{b}\phi(t)dt+2U(a)(b-a)}.
\end{equation}
\end{corollary}
\begin{proof}
From Theorem (\ref{t1}) we infer
\begin{eqnarray*}
\log I(b)-\log I(a)&=&\int_{a}^{b}(\log I)'(t)dt=2\int_{a}^{b}\phi(t)dt+2\int_{a}^{b}U(t)dt\\&\geq&2\int_{a}^{b}\phi(t)dt+2U(a)(b-a).
\end{eqnarray*}
\end{proof}

\subsection{Parabolic frequency on general parabolic inequality}
In the next we consider a more general parabolic inequality (\ref{a1}) and prove the following results.

\begin{theorem}\label{t2} Let $u: M\times [a,b] \to \mathbb{R}^n$  satisfies the following equation
\begin{equation}\label{a1}
|u_{t}-\mathcal{L} u|\leq \psi(t)\left(|u|+\sqrt{T(\nabla u, \nabla u)}\right)
\end{equation}
where  $\psi(t)$ is a time dependent function. Then
\begin{equation}\label{a2}
 \psi^{2}(t)\geq\left[\log (1-U(t))\right]'.
\end{equation}
\end{theorem}
\begin{proof}
We can write  $I'(t)$ as follows
\begin{equation}\label{a3}
I'(t)=2\int_{M}u\left(u_{t}-\frac{1}{2}(u_{t}-\mathcal{L}u)\right)d\mu+\int_{M}u\left(u_{t}-\mathcal{L}u\right)d\mu.
\end{equation}
Since
\begin{equation}
D(t)=\int_{M}u\left(u_{t}-\frac{1}{2}(u_{t}-\mathcal{L}u)\right)d\mu-\frac{1}{2}\int_{M}u\left(u_{t}-\mathcal{L}u\right)d\mu,
\end{equation}
we get
\begin{equation}\label{9}
I'(t)D(t)=2\left(\int_{M}u\left(u_{t}-\frac{1}{2}(u_{t}-\mathcal{L}u)\right)d\mu\right)^{2}-\frac{1}{2}\left(\int_{M}u\left(u_{t}-\mathcal{L}u\right)d\mu\right)^{2}.
\end{equation}
Differentiating $D(t)$ and using integration by parts we find
\begin{eqnarray*}
D'(t)&=&-2\int_{M}T(\nabla u,\nabla u_{t})d\mu\\&=&2\int_{M}u_{t}\mathcal{L}ud\mu\\&=&2\int_{M}\left\{
\left(u_{t}-\frac{1}{2}(u_{t}-\mathcal{L}u)\right)^{2}-\frac{1}{4}\left(u_{t}-\mathcal{L}u\right)^{2}
\right\}d\mu.
\end{eqnarray*}
Hence,
\begin{equation}\label{10}
D'(t)I(t)=2I(t)\int_{M}
\left(u_{t}-\frac{1}{2}(u_{t}-\mathcal{L}u)\right)^{2}d\mu-\frac{I(t)}{2}\int_{M}\left(u_{t}-\mathcal{L}u\right)^{2}
d\mu.
\end{equation}
Using the Cauchy-Schwarz inequality, (\ref{9}), (\ref{10}), and (\ref{a1}) we deduce
\begin{eqnarray}\nonumber
&&D'(t)I(t)-D(t)I'(t)\\\nonumber&=&2\left\{\int_{M}u^{2}d\mu\int_{M}
\left(u_{t}-\frac{1}{2}(u_{t}-\mathcal{L}u)\right)^{2}d\mu-\left(\int_{M}u\left(u_{t}-\frac{1}{2}(u_{t}-\mathcal{L}u)\right)d\mu\right)^{2}\right\}\\&&-\frac{I(t)}{2}\int_{M}\left(u_{t}-\mathcal{L}u\right)^{2}d\mu+\frac{1}{2}\left(\int_{M}u\left(u_{t}-\mathcal{L}u\right)d\mu\right)^{2}\\\nonumber&\geq&
-\frac{I(t)}{2}\int_{M}\left(u_{t}-\mathcal{L}u\right)^{2}d\mu\\\nonumber&\geq&
-\frac{\psi^{2}(t)I(t)}{2}\int_{M}\left(|u|+\sqrt{T(\nabla u, \nabla u)}\right)^{2}d\mu\\\nonumber&\geq&
-\psi^{2}(t)I(t)\left(I(t)-D(t)\right).
\end{eqnarray}
Dividing  both sides of last inequality  by $I^{2}(t)$ we obtain $U'(t)\geq \psi^{2}(t)\left(U(t)-1\right)$ and  this completes the proof.
\end{proof}

Applying Theorem \ref{t2} we can prove a generalization of Corollary \ref{cor1} as follows.
\begin{corollary}\label{cor2}
If $u: M\times [a,b] \to \mathbb{R}^n$  satisfies (\ref{a1}), then
\begin{equation*}
I(b)\geq I(a)\exp\left\{(b-a)\left( (2+\mathop{\sup}\limits_{[a,b]}\psi)\exp(\int_{a}^{b}\psi^{2}(t)dt)(U(a)-1)+2-\mathop{\sup}\limits_{[a,b]}\psi\right) \right\}.
\end{equation*}
In particular, if $u(\cdot,b)=0$ then $u=0$.
\end{corollary}
\begin{proof}
From (\ref{a3}) we have
\begin{eqnarray*}
(\log I)'(t)& =&2U(t)+\frac{2}{I(t)}\int_{M}u\left(u_{t}-\mathcal{L}u\right)d\mu\\
&\geq& 2U(t)-\frac{2\psi(t)}{I(t)}\int_{M}|u|\left(|u|+\sqrt{T(\nabla u, \nabla u)}\right)d\mu\\
&\geq&2U(t)-2\psi(t)-\frac{2\sqrt{2}\psi(t)}{I(t)}\int_{M}|u|\sqrt{T(\nabla u, \nabla u)}d\mu\\
&\geq&2U(t)-2\psi(t)(1+\sqrt{-U(t)})\geq (2+\psi(t))U(t)-3\psi(t).
\end{eqnarray*}
Thus, by taking integration we get
\begin{equation}\label{azs2}
\log I(b)-\log I(a)=\int_{a}^{b}(\log I)'(t)dt\geq (2+\mathop{\sup}\limits_{[a,b]}\psi)\int_{a}^{b}U(t)dt-3\mathop{\sup}\limits_{[a,b]}\psi (b-a).
\end{equation}
For $t\in[a,b]$, using (\ref{a2})  we obtain
\begin{equation}\nonumber
\log(1-U(t))\leq (1-U(a))+\int_{a}^{t}\psi^{2}(s)ds\leq (1-U(a))+\int_{a}^{b}\psi^{2}(s)ds.
\end{equation}
Hence,
\begin{equation}\label{azs3}
U(t))\geq (U(a)-1)\exp\left( \int_{a}^{b}\psi^{2}(s)ds\right)+1.
\end{equation}
Substituting (\ref{azs3}) in (\ref{azs2}), we deduce
\begin{equation}\nonumber
\log I(b)-\log I(a)\geq(b-a)\left( (2+\mathop{\sup}\limits_{[a,b]}\psi)\exp(\int_{a}^{b}\psi^{2}(t)dt)(U(a)-1)+2-\mathop{\sup}\limits_{[a,b]}\psi\right).
\end{equation}
This completes the proof of corollary.
\end{proof}
\begin{remark}
If we consider  $T=id$, then we obtain results of \cite{TH}.
\end{remark}

\section{p-parabolic frequency}\label{S3}
In this section we  consider the following nonlinear weighted $p$-heat-type equation for $1<p<\infty$
\begin{equation}\label{sa1}
|u(x,t)|^{p-2}u_{t}(x,t)-\Delta_{f,p} u(x,t)=\eta(t)|u(x,t)|^{p-2}u(x,t),
\end{equation}
 where $\eta$  is a smooth function with respect to time-variable $t$, $u$  is a  smooth function on $M\times [a,b]$.
Here, the parabolic frequency  $U_p(t)$ for the  solution of  (\ref{sa1})  is defined as follows
 \begin{equation}\label{sa2}
U_{p}(t)=\frac{D_{p}(t)}{I_{p}(t)},
\end{equation}
where
 \begin{equation}\label{sa4}
D_{p}(t)=\int_{M} u\Delta_{f,p}u d\mu=-\int_{M}|\nabla u|^{p}d\mu,
\end{equation}
and
 \begin{equation}\label{sa5}
I_{p}(t)=\int_{M}|u|^{p}d\mu
\end{equation}
with $d\mu = e^{-f}d\nu$ as defined in the introduction.
Our main results here is stated in the next theorem and we give an application in Corollary \ref{cor3}.
\begin{theorem}\label{t3}
Let $u: M\times [a,b] \to \mathbb{R}^n$  solves the heat equation (\ref{sa1}), then $U_p(t)$ is a nondecreasing function. Moreover, if $U_{p}(t)$ is constant, then $u$ is an eigenfunction of $f$  and
\begin{equation}\label{azs1}
u(x,t)|u(x,t)|^{p-2}=e^{(p-1)(U(t)t-U(a)a)}u(x,a)|u(x,a)|^{p-2}\exp\{(p-1)\int_{a}^{s}\eta(s)ds\}.
\end{equation}
Also, if $\eta(t)$ is a nondecreasing function then $\log I$ is convex.
\end{theorem}
\begin{proof}
By a straightforward computation we have 
  \begin{equation}\label{sa03}
\partial_{t}|u|^{p}=p|u|^{p-2}uu_t,
\end{equation}
  \begin{equation}\label{sa3}
\partial_{t}|\nabla u|^{p}=p|\nabla u|^{p-2}\langle\nabla u, \nabla u_{t}\rangle.
\end{equation}
So we calculate the derivative of $I_{p}(t)$ as follows
\begin{eqnarray}\label{sa6}
I'_{p}(t)&=&\int_{M}p|u|^{p-2}uu_{t}d\mu.
\end{eqnarray}
Noting  that by (\ref{sa4})
  \begin{equation}\label{sa7}
\int_{M}|\nabla u|^{p}d\mu=-\int_{M} u \Delta_{p}u d\mu=-D_{p}(t).
\end{equation}
Multiplying both sides of (\ref{sa1}) by $u$ and integrate over $M$ with respect to $d\mu$ gives
  \begin{equation}\label{sa8}
\int_{M}(|u|^{p-2}u_{t}-\Delta_{f,p} u)ud\mu= \int_{M} \eta(t)|u|^{p}d\mu=\eta(t)I_{p}(t).
\end{equation}
Clearly combining (\ref{sa7}) and (\ref{sa8}) yields
\begin{align*}
\int_{M} |u|^{p-2}u_{t} ud\mu= \eta(t)I_{p}(t) + D_p(t)
\end{align*}
from where we conclude by reverting to (\ref{sa6}) that
  \begin{equation}\label{sa9}
I'_{p}(t)=p\eta(t)I_{p}(t)+pD_{p}(t).
\end{equation}
Now, we compute the derivative of $D_{p}(t)$. Applying (\ref{sa3}) and integration by parts, we  get
\begin{eqnarray}\nonumber
D'_{p}(t)&=& -\int_{M}p|\nabla u|^{p-2}\langle\nabla u, \nabla u_{t}\rangle d\mu\\\label{sa10}&=&
p\eta(t)D_{p}(t)+p\int_{M}|u|^{2-p}(\Delta_{f,p} u)^{2}d\mu.
\end{eqnarray}
Using (\ref{sa9}), (\ref{sa10}), and  the Cauchy-Schwarz inequality,  we arrive at
\begin{eqnarray*}\nonumber
&&D'_{p}(t)I_{p}(t)-D_{p}(t)I'_{p}(t)\\&=&\left( p\eta(t)D_{p}(t)+p\int_{M}|u|^{2-p}(\Delta_{f,p} u)^{2}d\mu\right)I_{p}(t)\\&&-D_{p}(t)\left(p\eta(t)I_{p}(t)+pD_{p}(t)\right)\\&=&
p\left(\int_{M}|u|^{2-p}(\Delta_{f,p} u)^{2}d\mu\right)\left(\int_{M}|u|^{p}d\mu\right)-p\left(\int_{M}u\Delta_{f,p}ud\mu\right)^{2}\geq0
\end{eqnarray*}
since $p>1$.
Thus we have obtained 
\begin{equation*}
U'_{p}(t)=\frac{D'_{p}(t)I_{p}(t)-D_{p}(t)I'_{p}(t)}{I_{p}^{2}(t)}\geq0
\end{equation*}
which  implies that  $U$ is hence nondecreasing. If $U$ is constant then $U'_{p}(t)=0$  and equality in the Cauchy-Schwarz inequality implies that $\Delta_{f,p}u=\alpha(t) u|u|^{p-2}$ for some time-dependent smooth function $\alpha$. Thus, $u$ is an eigenfunction of $\Delta_{f,p}$. In this case, we have
\begin{equation*}
D_{p}(t)=\int_{M}u\Delta_{f,p} ud\mu=\alpha(t)I_{p}(t),
\end{equation*}
and consequently, $\alpha(t)=U(t)$, that is  $\Delta_{f,p}u=U(t) u$. Let $$w(x,t)=e^{-(p-1)Ut}u(x,t)|u(x,t)|^{p-2}.$$
By taking derivative  of $w(x,t)$ with respect to $t$ we obtain
\begin{equation*}
w_{t}(x,t)=(p-1)e^{-(p-1)Ut}\left(-U u|u|^{p-2}+u_{t}|u|^{p-2}\right)=(p-1)\eta(t)w(x,t)
\end{equation*}
Solving the last equation gives (\ref{azs1}).

\end{proof}
The last theorem has the following as an immediate consequence.
\begin{corollary}\label{cor3}
If $u: M\times [a,b] \to \mathbb{R}^n$  satisfies in (\ref{sa1}) then
\begin{equation}
I_{p}(b)\geq I_{p}(a)e^{p\int_{a}^{b}\eta(t)dt+pU_{p}(a)(b-a)}.
\end{equation}
\end{corollary}
\begin{proof}
From Theorem (\ref{t3}) we infer
\begin{eqnarray*}
\log I_{p}(b)-\log I_{p}(a)&=&\int_{a}^{b}(\log I_{p})'(t)dt=p\int_{a}^{b}\eta(t)dt+p\int_{a}^{b}U_{p}(t)dt\\&\geq&p\int_{a}^{b}\eta(t)dt+pU_{p}(a)(b-a).
\end{eqnarray*}
The proof is therefore complete. 
\end{proof}

\subsection{General parabolic operators}
In the next we consider a more general parabolic inequality for $p>1$
\begin{equation}\label{314}
|(|u|^{p-2}u_{t}- \Delta_{f,p} u)|\leq \psi(t)\left[|u|^{p-1}(1+|u|^{-\frac{p}{2}}\sqrt{|\nabla u|^p})\right],
\end{equation}
where  $\psi(t)$ is a time dependent function, and prove the following results.

\begin{theorem} Let $u: M\times [a,b] \to \mathbb{R}^n$  satisfies (\ref{314}). Then
\begin{align}\label{315}
U'(t) & \ge  \frac{p}{2}\psi^2(t)[U_p(t)-1],
\end{align}
\begin{align}\label{316}
 \psi^{2}(t)& \ge \frac{2}{p}\left[\log (1-U_p(t))\right]'
\end{align}
and 
\begin{align}\label{317}
[\log I(t)]' \ge  p(1+\psi(t)/2)U_p(t) - (3p/2)\psi(t).
\end{align}
\end{theorem}

\begin{proof}
Note that all the integrals in this subsection are over $M$ with respect to the weighted measure $d\mu=e^{-f}d\nu$.  First we write the expression for the derivative of $I_p(t)$ as follows:
\begin{align}\label{318}
I_p'(t) & = p\int |u|^{p-2}uu_t \nonumber\\
& = p\int u\Delta_{f,p} u + p\int \left\langle u,  (|u|^{p-2}u_t -  \Delta_{f,p} u) \right\rangle  \nonumber\\
& = p\int \left\langle u,  \left[|u|^{p-2}u_t -  \frac{1}{2} ( |u|^{p-2}u_t-  \Delta_{f,p} u)\right] \right\rangle + \frac{p}{2} \int \left\langle u, (|u|^{p-2}u_t - \Delta_{f,p} u) \right\rangle.
\end{align}
The quantity $D_p(t)$ can also be re-written as follows
\begin{align}\label{319}
D_p(t) & = \int \langle u, \Delta_{f,p} u\rangle \nonumber\\
& = \int  \left\langle u,  [|u|^{p-2}u_t -  \frac{1}{2} ( |u|^{p-2}u_t-  \Delta_{f,p} u)] \right\rangle - \frac{1}{2} \int \langle u, (|u|^{p-2}u_t - \Delta_{f,p} u) \rangle.
\end{align}
By (\ref{318}) and (\ref{319}) we obtain 
\begin{align}\label{320}
D_p(t)I_p'(t) & = p\left(\int  \langle u,  [|u|^{p-2}u_t -  \frac{1}{2} ( |u|^{p-2}u_t-  \Delta_{f,p} u)] \rangle \right)^2  \nonumber \\
& \hspace{1cm} -\frac{p}{2}\left(\int \langle u, (|u|^{p-2}u_t - \Delta_{f,p} u) \rangle \right)^2.
\end{align}
Similarly, differentiating the quantity $D_p(t)$ and rewrite the resulting expression gives
\begin{align}\label{321}
D_p'(t) & = -p\int |\nabla u|^{p-2}\langle\nabla u, \nabla u_t\rangle  = p\int \langle u_t,  \Delta_{f,p} u \rangle \nonumber\\
& = p \int \langle u_t, \left[|u|^{p-2}u_t - \left( |u|^{p-2}u_t - \Delta_{f,p} u \right) \right]\rangle \nonumber\\
& = p \int \Bigg\{ \Bigg| |u|^{\frac{p}{2}-1}u_t - \frac{1}{2} \left[ |u|^{1-\frac{p}{2}} \left(|u|^{p-2}u_t -\Delta_{f,p} u \right)\right]\Bigg|^2  \nonumber \\
& \hspace{1cm} - \frac{1}{4} \Bigg| |u|^{1-\frac{p}{2}}\left( |u|^{p-2}u_t -\Delta_{f,p} \right)\Bigg|^2   \Bigg\}.
\end{align}
Therefore we have by (\ref{321})
\begin{align}\label{322}
I_p(t)D_p'(t) & = pI_p(t) \int \Bigg| |u|^{\frac{p}{2}-1}u_t - \frac{1}{2} \left[ |u|^{1-\frac{p}{2}} \left(|u|^{p-2}u_t -\Delta_{f,p} u \right)\right]\Bigg|^2 \nonumber \\
& \hspace{1cm} - \frac{pI_p(t)}{4}  \int \Bigg| |u|^{1-\frac{p}{2}}\left( |u|^{p-2}u_t -\Delta_{f,p} \right)\Bigg|^2.
\end{align}
Combining (\ref{320}) and (\ref{322}) we have 
\begin{align*}
 I_p^2(t) U_p'(t)  =  &  \ I_p(t)D_p'(t) - D_p(t)I_p'(t)\\
  = & \  p\left(\int|u|^p\right) \int \Bigg| |u|^{\frac{p}{2}-1}u_t - \frac{1}{2} \left[ |u|^{1-\frac{p}{2}} \left(|u|^{p-2}u_t -\Delta_{f,p} u \right)\right]\Bigg|^2\\
 & - p \left(\int  \langle u,  [|u|^{p-2}u_t -  \frac{1}{2} ( |u|^{p-2}u_t-  \Delta_{f,p} u)] \rangle \right)^2 \\
 &  - \frac{pI_p(t)}{4}  \int \Bigg| |u|^{1-\frac{p}{2}}\left( |u|^{p-2}u_t -\Delta_{f,p} \right)\Bigg|^2 \\
 & + \frac{p}{2}\left(\int \langle u, (|u|^{p-2}u_t - \Delta_{f,p} u) \rangle \right)^2\\
 \geq  & \   - \frac{pI_p(t)}{4}  \int \Bigg| |u|^{1-\frac{p}{2}}\left( |u|^{p-2}u_t -\Delta_{f,p} \right)\Bigg|^2 \\
 & + \frac{p}{2}\left(\int \langle u, (|u|^{p-2}u_t - \Delta_{f,p} u) \rangle \right)^2\\
 \geq & \ - \frac{pI_p(t)}{4}  \int \Bigg| |u|^{1-\frac{p}{2}}\left( |u|^{p-2}u_t -\Delta_{f,p} \right)\Bigg|^2,
\end{align*}
where we have used the Cauchy-Schwarz inequality and the fact that $p>1$.

Now applying the general parabolic inequality (\ref{314}) leads us to
\begin{align*}
 I_p^2(t) U_p'(t) & \geq - \frac{p}{4} I_p(t)\psi^2(t)  \int |u|^{2-p}  \Bigg| \left( |u|^{p-1} +|u|^{\frac{p}{2}-1} \sqrt{|\nabla u|^p} \right)\Bigg|^2\ d\mu\\
 & \geq  - \frac{p}{4}  I_p(t)\psi^2(t) \int   \Big| \left( \sqrt{|u|^p}  + \sqrt{|\nabla u|^p} \right)\Big|^2\ d\mu\\
 &  \geq  - \frac{p}{2} \psi^2(t)   I_p(t) \Big(I_p(t) - D_p(t)\Big),
\end{align*}
where we have used the elementary inequality of the form $(a+b)^2 \le 2(a^2+b^2)$.  Dividing both sides by $I_p^2(t)$ yields the first claim, that is, (\ref{315}). Clearly (\ref{316}) follows from (\ref{315}).

Finally, we prove the last part of the theorem (i.e., (\ref{317})).  Here we write the expression for $I_p'(t)$ as follows
\begin{align}\label{323}
I_p'(t) & = p\int |u|^{p-2}u u_t \ d\mu \nonumber\\
& =  p\int_M u\Delta_{f,p} u \ d\mu + p\int_M \langle u,  (|u|^{p-2}u_t -  \Delta_{f,p} u) \rangle  d\mu \nonumber\\
&\geq pD_p(t) - p \psi(t) \int_M |u| \left[ |u|^{p-1} + |u|^{\frac{p}{2}-1} |\nabla u|^{\frac{p}{2}}\right]d\mu \nonumber\\
& = pD_p(t) - p\psi(t)I_p(t) - p\psi(t) \int_M   |u|^{\frac{p}{2}} |\nabla u|^{\frac{p}{2}}\ d\mu.
\end{align}
By the Cauchy-Schwarz inequality we know that
\begin{align}\label{324}
\int_M   |u|^{\frac{p}{2}} |\nabla u|^{\frac{p}{2}}\ d\mu & \le \sqrt{\left(\int_M |u|^p\ d\mu\right)} \sqrt{\left(\int_M |\nabla u|^p\ d\mu\right)}  \nonumber\\
& = \sqrt{I_p(t)} \sqrt{-D_p(t)}.
\end{align}
Combining (\ref{323}) and (\ref{324}) we have 
\begin{align*}
I_p'(t) \geq pD_p(t) - p\psi(t)\Big( I_p(t) + \sqrt{I_p(t)} \sqrt{-D_p(t)}\Big).
\end{align*}
Hence
\begin{align}\label{325}
[\log I_p(t)]' =  I_p'(t)/I_p(t) \ge pU_p(t) - p\psi(t)\Big( 1 + \sqrt{-U_p(t)}\Big).
\end{align}
Using the elementary inequality of the form $a\le \frac{1}{2}(a^2+1)$ applied to $a = \sqrt{-U_p(t)}$ in (\ref{325}), we have
\begin{align*}
[\log I_p(t)]' & \ge pU_p(t) - p\psi(t)\Big[ 1 + \frac{1}{2} (-U_p(t) +1)\Big]\nonumber\\
&  = p\left( 1+ \frac{\psi(t)}{2} \right) U_p(t) - \frac{3p}{2}\psi(t).
\end{align*}
This completes the proof
\end{proof}

The next result is an immediate consequence of the above theorem.
\begin{theorem}
Let $u: M\times [a,b] \to \mathbb{R}^n$  satisfies (\ref{314}). Then 
\begin{equation}\label{aa}
I_p(b) \geq I_p(a)\exp\left\{\frac{p(b-a)}{2} \Lambda\right\},
\end{equation}
where 
\begin{align*}
\Lambda := [U_p(a)-1]\exp\left\{ [2+\sup_{[a,b]}\psi(t)]\frac{p}{2}\int_a^b \psi^2(t)dt \right\} + 1-3 \sup_{[a,b]}\psi(t).
\end{align*}
In particular, if $u(\cdot,b)=0$, then $u\equiv 0$.
\end{theorem}

\begin{proof}
Starting with (\ref{317}) we get
\begin{align}\label{6a}
\log I_p(b)-\log I_p(a) &= \int_a^b (\log I)'(s) ds\nonumber\\
& \geq [2+\sup_{[a,b]}\psi(t)] \frac{p}{2} \int_a^bU_p(s)ds - \frac{3p}{2} \sup_{[a,b]}\psi(t) (b-2).
\end{align}
From (\ref{316}) we obtain for $r\in [a,b]$
\begin{align*}
\log [1-U_p(r)]\le \log [1-U_p(a)] + \frac{p}{2} \int_a^b \psi^2(r)dr.
\end{align*}
Thus
\begin{align}\label{6b}
U_p(r)\geq \exp \left( \frac{p}{2} \int_a^b \psi^2(t)dt \right) [U_p(a)-1] +1.
\end{align}
Inserting the estimate (\ref{6b}) into  (\ref{6a}) and integrating the resulting expression gives
\begin{align*}
\log I_p(b) - \log I_p(a) \ge  & \  \frac{p(b-a)}{2} \Bigg[[U_p(a)-1]\exp\left\{ [2+\sup_{[a,b]}\psi(t)]\frac{p}{2}\int_a^b \psi^2(t)dt \right\} \\
& + 1-3 \sup_{[a,b]}\psi(t) \Bigg]
\end{align*}
from where the estimate (\ref{aa}) follows. This completes the proof.
\end{proof}


This work  does not receive any funding.

 \subsection*{Conflict of interests}
 We declare that we do not have any commercial or associative interest that represents
a conflict of interest in connection with the work submitted.





\end{document}